\documentclass[12pt, reqno]{amsart}
\usepackage[all]{xy}

\usepackage{amssymb}
\usepackage{amsthm}
\usepackage{amsmath,cancel}
\usepackage{amscd,enumitem}
\usepackage{verbatim}
\usepackage{float}
\usepackage{color,mdframed}
\usepackage{bbm}
\usepackage[mathscr]{eucal}
\usepackage[all]{xy}
\usepackage{bm}
\usepackage[hang,flushmargin]{footmisc}

\usepackage{a4wide,fullpage}  
\usepackage[draft]{hyperref}
\usepackage{mathtools}
\usepackage{calligra}
\usepackage{stmaryrd}
\DeclareMathAlphabet{\mathcalligra}{T1}{calligra}{m}{n}

\newtheorem{theorem}{Theorem}

\newtheorem{lemma}[theorem]{Lemma}

\newtheorem{proposition}[theorem]{Proposition}
\newtheorem*{theorem*}{Theorem}

\theoremstyle{definition}

\theoremstyle{remark}
\newtheorem*{remark}{Remark}
\newtheorem*{remarks}{Remarks}

\newcommand{\R}{\mathbb{R}}
\newcommand{\Q}{\mathbb{Q}}
\newcommand{\Z}{\mathbb{Z}}
\newcommand{\N}{\mathbb{N}}
\newcommand{\C}{\mathbb{C}}

\renewcommand{\H}{\mathbb{H}}

\newcommand{\F}{\mathbb{F}}

\newcommand{\ord}{{\text {\rm ord}}}

\newcommand{\tr}{\operatorname{tr}}

\newcommand{\sgn}{\operatorname{sgn}}

\newcommand{\SL}{{\text {\rm SL}}}

\renewcommand{\pmod}[1]{\  \,  \left(  \operatorname{mod} \,  #1 \right)}

\DeclareMathAlphabet{\mathpzc}{OT1}{pzc}{m}{it}

\numberwithin{equation}{section}
\numberwithin{theorem}{section}
\allowdisplaybreaks

\begin{document}
	
	\vspace*{-1cm}
	
\title[Moments of class numbers in arithmetic progressions]{Distribution of moments of Hurwitz class numbers in arithmetic progressions and holomorphic projection}
%\author{Kathrin Bringmann}
%\address{University of Cologne, Department of Mathematics and Computer Science, Weyertal 86-90, 50931 Cologne, Germany}
%\email{kbringma@math.uni-koeln.de}
\author{Ben Kane}
\address{Department of Mathematics, University of Hong Kong, Pokfulam, Hong Kong}
\email{bkane@hku.hk}
\author{Sudhir Pujahari}
\address{School of Mathematical Sciences, National Institute of Science Education and Research, Bhubaneswar, An OCC of Homi Bhabha National Institute,  P. O. Jatni,  Khurda 752050, Odisha, India.}
\email{ spujahari@niser.ac.in/ spujahari@gmail.com}
\thanks{The research of the first author was supported by grants from the Research Grants Council of the Hong Kong SAR, China (project numbers HKU 17301317, and 17303618).  Most of the research was conducted while the second author was a postdoctoral fellow at The University of Hong Kong. He thanks the university for providing excellent facilities and a productive environment and would also like to thank the Institute of Mathematical Research at the University of Hong Kong for providing partial financial assistance to the postdoctoral fellowship.}
\date{\today}
\begin{abstract}
In this paper, we study moments of Hurwitz class numbers associated to imaginary quadratic orders restricted into fixed arithmetic progressions. 
In particular, we fix $t$ in an arithmetic progression $t\equiv m\pmod{M}$ and consider the ratio of the $2k$-th moment to the zeroeth moment for $H(4n-t^2)$ as one varies $n$. The special case $n=p^r$ yields as a consequence asymptotic formulas for moments of the trace $t\equiv m\pmod{M}$ of Frobenius on elliptic curves over finite fields with $p^r$ elements.
\end{abstract}
\maketitle

\section{Introduction and statement of results}\label{sec:intro}
Class numbers of imaginary quadratic orders and the corresponding class numbers of integral binary quadratic forms have a long and storied history going back at least to Gauss, who asked for a classification of all binary quadratic forms with a given class number. Dirichlet's class number formula relates the class numbers with $L$-functions, another central area of study, and the growth of class numbers is in turn closely related to the generalized Riemann hypothesis; see \cite{BhandMurty} for a detailed history and an introduction to the development of the subject.

There are many identities and relations between sums of class numbers and interesting arithmetic functions. To state one such interesting identity that motivates our study in this paper, let $\mathcal{Q}_D$ denote the set of integral binary quadratic forms of discriminant $D<0$. The \begin{it}$|D|$-th Hurwitz class number\end{it} is defined by
\[
H(|D|):=\sum_{Q\in\mathcal{Q}_{D}/\SL_2(\Z)} \frac{1}{\omega_Q},
\]
where $\omega_Q$ is half the size of the stabilizer group $\Gamma_Q$ of $Q$ in $\SL_2(\Z)$. By convention, we set $H(0):=-\frac{1}{12}$ and $H(r):=0$ for $r\notin\N_0$ or $r\equiv 1,2\pmod{4}$. For a prime $p$, a famous identity relating sums of these class numbers with an elementary function is given by  (see \cite[p. 154]{Eichler})
\[
\sum_{t\in\Z} H\left(4p-t^2\right) = 2p.
\]
In \cite{BrownCalkin}, similar identities were studied in the case when $t$ is restricted to a given arithmetic progression, leading to further investigation of such sums of class numbers. For example, identities such as 
\[
\sum_{\substack{t\in \Z\\ t\equiv 1\pmod{5}}} H\left(4p-t^2\right)=\begin{cases} 
\frac{1}{3}(p+1)&\text{if }p\equiv 1,2\pmod{5},\\
\frac{1}{2}(p-1) &\text{if }p\equiv 3\pmod{5},\\ 
\frac{5}{12}(p+1)&\text{if }p\equiv 4\pmod{5}
\end{cases}
\]
were proven in \cite{BKClassNum} and \cite{BrownCalkin}. In this vein, we further consider sums of moments of these Hurwitz class numbers defined by 
\begin{equation}\label{eqn:HmMkdef}
H_{\kappa,m,M}(n):=\sum_{\substack{t\in\Z\\t\equiv m \pmod{M}}}t^{\kappa} H \left(4n-t^2\right).
\end{equation}
In the following we drop the condition $t\in\Z$ in the summation and throughout we omit $\kappa$ from the notation whenever $\kappa=0$. Moments of class numbers have occurred throughout the literature (for example, see \cite{BarbanGordover,Birch,Lavrik,Wolke1,Wolke2,Wolke3}). The main result of this paper is an asymptotic formula for the even moments in an arbitrary arithmetic progression, given in terms of the the zeroeth moment.  
\begin{theorem}\label{thm:HmMk}
Let $m,M,k\in \N$ be given. As $n\to\infty$, we have 
\[
\frac{H_{2k,m,M}(n)}{n^k H_{m,M}(n)} = C_k +O_{k,M,\varepsilon}\left(n^{-\frac{1}{2}+\varepsilon}\right).
\]
\end{theorem}
Due to a well-known relation between class numbers and the trace of Frobenius of elliptic curves from the influential works of Deuring \cite{Deuring}, Birch \cite{Birch}, and Schoof \cite{Schoof}, Theorem \ref{thm:HmMk} yields a corollary about moments related to the trace of Frobenius on elliptic curves over finite fields.  For an elliptic curve $E$ defined over the finite field $\F_{p^r}$ with $p^r$ elements ($p$ prime, $r\in\mathbb N$), the \begin{it}trace of Frobenius\end{it} is given by 
\begin{equation*}%\label{eqn:Frobenius}
\operatorname{tr}(E)=\operatorname{tr}_{p^r}(E):=p^r+1-\#E\left(\F_{p^r}\right). 
\end{equation*}
Here $E(\F_{p^r})$ is the set of points on the elliptic curve over the finite field $\F_{p^r}$. Our restriction on $t$ in \eqref{eqn:HmMkdef} corresponds to restricting the elliptic curve $E/\F_{p^r}$ to  the set 
\begin{equation*}%\label{eqn:EmMndef}
\mathcal{E}_{m,M,p^r}:=\{E/\F_{p^r}: \operatorname{tr}(E) \equiv m\pmod{M}\}.
\end{equation*}
The \begin{it}weighted $\kappa$-th moment with respect to $\operatorname{tr}(E)$\end{it} (for $\kappa\in\N_0$) is then given by 
\begin{equation}\label{eqn:SmMkdef}
S_{\kappa,m,M}(p^r):=\sum_{\substack{E/\F_{p^r}\\ \operatorname{tr}(E)\equiv m\pmod{M}}} \frac{\operatorname{tr}(E)^{\kappa}}{\#\operatorname{Aut}_{\F_{p^r}}(E)}=\sum_{E\in\mathcal{E}_{m,M,p^r}} \frac{\operatorname{tr}(E)^{\kappa}}{\#\operatorname{Aut}_{\F_{p^r}}(E)}.
\end{equation}
Before stating our first result, we discuss how $S_{\kappa,m,M}(p^r)$ is expected to grow as $p^r\to\infty$. For simplicity, we consider the case $n=p$ prime in this heuristic argument. By the Hasse bound \cite{Hasse}, conjectured by Artin~\cite{Artin} in his thesis, we have that $|\operatorname{tr}(E)|\leq 2\sqrt{p}$. Taking $-1\leq a\leq b\leq 1$ and a fixed elliptic curve $E$ over $\Q$ and considering the reduction of $E$ to elliptic curves over finite fields, it was independently conjectured by Sato and Tate (see e.g. \cite{Harris}) that if $E$ does not have complex multiplication, then
\[
\lim_{N\to\infty} \frac{\#\{p\leq N: 2a \sqrt{p}\leq  \operatorname{tr}(E)\leq 2 b \sqrt{p}\}}{\#\{p\leq N\}}=\frac{2}{\pi}\int_{a}^{b} \sqrt{1-x^2} dx.
\]
This was later proven in a series of collaborations between Barnet-Lamb, Clozel, Geraghty, Harris, Shepherd-Barron, and Taylor \cite{BGHT,CHT,HST}. Motivated by the Sato--Tate conjecture, Birch instead fixed a prime $p$ and varied $E/\F_{p}$, leading to an investigation of the sums $S_{\kappa,0,1}(p)$.  If there is no cancellation (this holds automatically for $\kappa$ even because every term in \eqref{eqn:SmMkdef} is non-negative), then one expects that for some constant $\mathscr{D}_{\kappa}$
\[
S_{\kappa,m,M}\left(p^r\right)\sim \mathscr{D}_{\kappa} \sum_{\substack{E/\F_{p^r}\\ \operatorname{tr}(E)\equiv m\pmod{M}}} \frac{p^{\frac{r\kappa}{2}}}{\#\operatorname{Aut}_{\F_{p^r}}(E)}= \mathscr{D}_{\kappa}p^{\frac{r\kappa}{2}}  S_{m,M}\left(p^r\right).
\]
Birch \cite{Birch} proved that these constants $\mathcal{D}_{\kappa}$ exist for $M=1$ and $r=1$ and moreover showed that they agree with the moments of the Sato--Tate distribution, proving the Sato--Tate conjecture for this family of elliptic curves. This work was influential, leading to a number of generalizations (for example, see \cite{AS,BrockGranville,DKS,DLZ,GM,Gekeler,KatzSarnak,McKee,SZ}).

In an orthogonal directiton, Castryck and Hubrechts \cite{CastryckHubrechts} studied the probability that $t$ lies in a given arithmetic progression, which is closely related to $S_{m,M}(p^r)$. Motivated by \cite{CastryckHubrechts} and \cite{Birch}, we generalize Birch's result by restricting the family of elliptic curves to $\mathcal{E}_{m,M,p^r}$. Specifically, we show that for $\kappa=2k \in 2\N$ the constant $\mathcal{D}_{2k}$  does indeed exist and equals the $k$-th Catalan number $C_k$ if $p$ is fixed and $r\to \infty$ or if $r\leq 2$ is fixed and $p\to\infty$.
\begin{theorem}\label{thm:SmMk}
Let $m\in\Z$, $M\in\N$ and $\varepsilon > 0$ be given. 
\noindent

\noindent
\begin{enumerate}[leftmargin=*,label={\rm(\arabic*)}]
\item
For primes $p \to \infty$, we have 
\begin{align*}
\frac{S_{2k,m,M}(p)}{p^k S_{m,M}(p)}=C_k+O_{k,M,\varepsilon}\left(p^{-\frac{1}{2}+\varepsilon}\right), \quad
\frac{S_{2k,m,M}\left(p^r\right)}{p^{rk}S_{m,M}\left(p^r\right)}=C_k+O_{k,M,r,\varepsilon}\left(p^{-1+\varepsilon}\right)\quad (r\geq 2). 
\end{align*}
\item 
Let $p>3$ be a prime for which $p\nmid \gcd(m,M)$ and $k\in\N$. As $r\to \infty$, we have 
\[
\frac{S_{2k,m,M}\left(p^r\right)}{p^{rk}S_{m,M}\left(p^r\right)}=C_k+O_{k,p,M,\varepsilon}\left(p^{\left(-\frac{1}{2}+\varepsilon\right)r}\right). 
\]
\end{enumerate}

\end{theorem}
\begin{remarks}
\noindent

\noindent
\begin{enumerate}[leftmargin=*,label={\rm(\arabic*)}]
\item
For $M=1$, these sums were studied by Birch \cite{Birch} and implicitly appear in the work of Ihara \cite{Ihara} (see also \cite[Theorem 1, Theorem 2]{KaplanPetrow2}). They obtained a formula for these sums in terms of the trace of Hecke operators that yields the asymptotic obtained in Theorem \ref{thm:SmMk}.
For $M=2$, formulas for $S_{2k,m,2}$ were obtained by Kaplan and Petrow (see for example \cite[Theorem 8]{KaplanPetrow}). 
\item
The implied constant in the error term is ineffective due to an ineffective lower bound in Lemma \ref{lem:HmM0lower} that uses Siegel's ineffective lower bound \cite{Siegel} for the class numbers of imaginary quadratic fields (see Lemma \ref{lem:Siegel} below). Using Littlewood's conditional effective bound for the class numbers \cite{Littlewood}, it can be made effective under the Generalized Riemann Hypothesis. Moreover, for fixed $M$ one should in principle be able to obtain an effective version of Lemma \ref{lem:HmM0lower} by computing the Eisenstein series components of certain modular forms; this was carried out for primes $M\leq 7$ in \cite[Section 4]{BKClassNum} and \cite[Corollary 7.3]{Me}.
\item A number of papers have investigated sums where the elliptic curves are restricted to those elliptic curves containing specified subgroups (see \cite[Section 6.2]{Kowalski} and \cite{KaplanPetrow2}). For $M$ squarefree, the special case $m=p^r+1$ overlaps with such results because $\tr(E)\equiv p^r+1\pmod{M}$ is equivalent to $M\mid \#E(\F_{p^r})$, which in turn implies that the set of \emph{$M$-torsion points}
\[
E[M]:=\{P\in E: \operatorname{ord}(P)\mid M\}
\]
is non-empty. Here $\operatorname{ord}(P)$ means the order of the point under the group law defined on elliptic curves. To sttate the special case of Theorem \ref{thm:SmMk} in this case, we denote the subset of torsion points of precise order $M$ by
\[
E^*[M]:=\{P\in E: \operatorname{ord}(P)= M\}
\]
and define 
\[
S_{\kappa,M}^*({p^r}):=\sum_{\substack{E/\F_{p^r}\\ E^*[M]\neq \emptyset}} \frac{\operatorname{tr}(E)^{\kappa}}{\#\operatorname{Aut}_{\F_{p^r}}(E)}.
\]
Then Theorem \ref{thm:SmMk} implies the following. 
\noindent
\begin{enumerate}[leftmargin=*,label={\rm(\arabic*)}]
\item As $p\to\infty$, we have 
\begin{align*}
\frac{S_{2k,M}^*\left(p\right)}{p^{k}S_{M}^*\left(p\right)}&=C_k+O_{k,M\varepsilon}\left(p^{-\frac{1}{2}+\varepsilon}\right),\qquad
\frac{S_{2k,M}^*\left(p^r\right)}{p^{rk}S_{M}^*\left(p^r\right)}=C_k+O_{k,M,r,\varepsilon}\left(p^{-1+\varepsilon}\right)\quad (r\geq 2). 
\end{align*}
\item 
If $p>3$ is a prime, then as $r\to \infty$ we have
\[
\frac{S_{2k,M}^*\left(p^r\right)}{p^{rk}S_{M}^*\left(p^r\right)}=C_k+O_{k,p,M,\varepsilon}\left(p^{\left(-\frac{1}{2}+\varepsilon\right)r}\right). 
\]
\end{enumerate}
\end{enumerate}
\end{remarks}

After some preliminary setup in Section \ref{sec:prelim}, we begin by investigating Hurwitz class numbers and then the moments in Section \ref{sec:holprojH}, proving Theorem \ref{thm:HmMk}. We then return to the application of these moments to elliptic curves in Section \ref{sec:EllCurveApplication}, proving Theorem \ref{thm:SmMk}.
\rm

\section*{Acknowledgements}
The authors are indebted to Kathrin Bringmann for many helpful suggestions and detailed discussions and also thank Yuk-Kam Lau for helpful discussion. 
\section{Preliminaries}\label{sec:prelim}

\subsection{Holomorphic and non-holomorphic modular forms}
We give a brief overview of the theory of modular forms here; for details, see \cite{Koblitz,OnoBook}. For $d$ odd, we set
\[
\varepsilon_{d}:=\begin{cases} 1 &\text{if }d\equiv 1\pmod{4}\hspace{-1pt},\\ i&\text{if }d\equiv 3\pmod{4}\hspace{-1pt}.\end{cases}
\]
\noindent Let $\Gamma$ be a congruence subgroup containing $T:=\left(\begin{smallmatrix}1&1\\ 0 &1\end{smallmatrix}\right)$ and if $\kappa \in \frac{1}{2}\in\mathbb Z$, then we also require that $\Gamma \subseteq \Gamma_0(4)$. 
 A function $F:\H\to\C$ satisfies \begin{it}modularity of weight $\kappa\in\frac{1}{2}\Z$ on $\Gamma\subseteq \SL_2(\Z)$ \end{it}\textit{with character $\chi$} if for every $\gamma=\left(\begin{smallmatrix}a&b\\ c&d\end{smallmatrix}\right)\in\Gamma$ we have 
\[
F|_{\kappa}\gamma  = \chi(d) F.
\]
Here the weight $\kappa$ \begin{it}slash operator\end{it} is defined by 
\[
F\big|_{\kappa}\gamma(\tau):= \left( \frac cd \right)^{2\kappa} \varepsilon_d^{2\kappa}(c\tau+d)^{-\kappa} F(\gamma\tau),
\]
 where $(\frac{\cdot}{\cdot})$ denotes the \begin{it}extended Legendre symbol\end{it}. We call $F$ a \begin{it}(holomorphic) modular form\end{it} if $F$ is holomorphic on $\H$ and $F(\tau)$ grows at most
 polynomially in $v$ as $\tau=u+iv\to \Q\cup\{i\infty\}$. To define certain non-holomorphic modular forms, let $\Delta_{\kappa}:=-v^2(\frac{\partial^2}{\partial u^2}+\frac{\partial^2}{\partial v^2})+i\kappa v(\frac{\partial}{\partial u}+i\frac{\partial}{\partial v}) $ be the \emph{weight $\kappa$ hyperbolic Laplace operator}.
A smooth function $F$ transforming modular of weight $\kappa$ is a \emph{harmonic Maass form of weight $\kappa$} if $\Delta_{\kappa}(F)=0$ and 
 there exists $a\in\R$ such that 
\begin{equation*}\label{eqn:fgrowth}
F(\tau)=O\left(e^{a v}\right)\text{ as }v\to \infty\qquad\text{ and }\qquad F(u+iv)=O\left(e^{\frac{a}{v}}\right)\text{ for }u\in\Q\text{ as }v\to 0^+.
\end{equation*}
If $F$ is moreover holomorphic on $\H$, then we call $F$ a \begin{it}weakly holomorphic modular form\end{it}.  For a harmonic Mass form $F$ of weight $\kappa$,  $\xi_{\kappa}(F)$ with $\xi_{\kappa}:=2iv^{\kappa} \overline{\frac{\partial}{\partial \overline{\tau}}}$ is a weakly holomorphic modular form of weight $2-\kappa$. 

 Suppose that $\kappa\neq 1$.
 Letting $\Gamma(\alpha,x)$ denote the incomplete gamma function, a harmonic Maass form of weight $\kappa$ on $\Gamma$ has a Fourier expansion of the form 
\[
F(\tau)=F^+(\tau)+F^{-}(\tau)
\]
with (for $q:=e^{2\pi i \tau }$)
\begin{align*}
F^+(\tau)&=\sum_{n\gg -\infty} c_F^+(n) q^n\\
F^-(\tau)&=c_F^-(0) v^{1-\kappa} + \sum_{0\neq n\ll \infty} c_F^-(n) \Gamma\left(1-\kappa,-4\pi n v\right) q^{n}.
\end{align*}
Another type of non-holomorphic modular form that naturally occurs is an \begin{it}almost holomorphic modular form\end{it}, which is a function $F:\H\to\C$ satisfying weight $\kappa$ modularity on $\Gamma$ for which there exist holomorphic functions $F_j$ ($0\leq j\leq \ell$) such that $F(\tau)=\sum_{j=0}^{\ell} F_j(\tau) v^{-j}$. We call $F_0$ a \begin{it}quasimodular form\end{it}.

There are natural operators that preserve modularity. In particular,  suppose that $F(\tau)=\sum_{n\geq n_0} c_{F,v}(n)q^n$ satisfies weight $\kappa$ modularity with Nebentypus character $\chi$ (of modulus $N$) on $\Gamma_0(N)\cap\Gamma_1(M)$ with $M\mid N$.  We have that
\[
F\big|V_\delta(\tau):=F(\delta\tau)
\]
satisfies weight $\kappa$ modularity on $\Gamma_0(\operatorname{lcm}(4,\delta N))\cap \Gamma_1(M)$ with Nebentypus $\chi\cdot(\frac{\delta}{\cdot})^{2k}$, and
\[
F\big|U_\delta(\tau):=\sum_{n\geq n_0} c_{F,\frac{v}{\delta}}(\delta n) q^n
\]
satisfies weight $\kappa$ modularity on $\Gamma_0(\operatorname{lcm}(4,N,\delta))\cap \Gamma_1(M)$ with Nebentypus $\chi\cdot(\frac{\delta}{\cdot})^{2k}$.

\subsection{Rankin-Cohen brackets}
For $F_1,F_2$ transforming like modular forms of weight  $\kappa_1,\kappa_2 \in \frac 12\Z$,  respectively, define  for $k \in \N_0$ the $k$-th {\it Rankin-Cohen bracket}
\begin{equation*}
[F_1,F_2]_k := \frac{1}{(2\pi i)^k}\sum_{j=0}^{k} (-1)^j \binom{\kappa_1 + k -1}{k-j} \binom{\kappa_2 + k -1}{j} F_1^{(j)} F_2^{(k-j)}
\end{equation*}
with  $\binom{\alpha}{j}:=\frac{\Gamma(\alpha+1)}{j!\Gamma(\alpha-j+1)}$. Then $[F_1,F_2]_k$ transforms modular form of weight $\kappa_1+\kappa_2+2k$.

\subsection{Elliptic curves and trace of Frobenius}

A good introduction to elliptic curves is \cite{Silverman}. 
For an elliptic curve $E$ defined over $\F_{p^r}$, we define the \begin{it}Frobenius endomorphism\end{it} $\operatorname{Fr}$ from $E$ to itself via the $p$-th power map. Namely, for a point $P=(X,Y)\in E$, we set
\[
\operatorname{Fr}(P):=\left(X^p,Y^p\right).
\]
The trace of Frobenius is given by $\operatorname{tr}(E)$. For $r=1$, Hasse \cite{Hasse} showed that 
\[
|\operatorname{tr}(E)|\leq 2\sqrt{p}.
\]
The distribution of $\frac{1}{2\sqrt{p}} \operatorname{tr}(E)$ has been well-studied and it is natural to group those elliptic curves whose trace of Frobenius agree. 
For $t\in\Z$, we hence define 
\begin{equation*}%\label{eqn:Etndef}
\mathcal{E}_{p^r,t}:=\{E/\F_{p^r}: \operatorname{tr}(E)=t\}.
\end{equation*}
As is well-known, there is a group law defined on elliptic curves and the automorphisms of the group we denote by $\operatorname{Aut}_{\F_{p^r}}(E)$. The automorphism group gives a natural weighting on the elliptic curves in $\mathcal{E}_{{p^r},t}$, leading to the definition
\begin{equation*}%\label{eqn:NAdef}
N_A(p^r;t):=\sum_{E\in\mathcal{E}_{{p^r},t}} \frac{1}{\#\operatorname{Aut}_{\F_{p^r}}(E)}.
\end{equation*}
These sums naturally occur when investigating $S_{\kappa,m,M}({p^r})$ because
\begin{equation*}\label{eqn:SmMksum}
S_{\kappa,m,M}({p^r}) = \sum_{t\equiv m\pmod{M}} t^{\kappa}N_A({p^r};t).
\end{equation*}

\subsection{The class number generating function}
Let 
\[
\mathcal{H}(\tau):=\sum_{n\in\Z} H(n) q^n
\]
be the generating function for the Hurwitz class numbers. Its modular properties follow by  \cite[Theorem 2]{HZ}.
\begin{theorem}\label{thm:Hcomplete}
The function 
\begin{equation*}
\widehat{\mathcal{H}}(\tau):=\mathcal{H}(\tau) +\frac{1}{8\pi \sqrt{v}}+ \frac{1}{4\sqrt{\pi}}\sum_{n=1}^\infty n\Gamma\left(-\frac12, 4\pi n^2 v\right)q^{-n^2}
\end{equation*}
is a harmonic Maass  form of weight $\frac{3}{2}$ on $\Gamma_0(4)$.
\end{theorem}

\subsection{Elliptic curves and class numbers} 

For $m\in\Z$, $n,M\in\N$, and $\kappa\in\N_0$, we next relate $S_{m,M,\kappa}(n)$ to certain sums of Hurwitz class numbers given by 
\[
\mathscr{H}_{\kappa,m,M}(p;n):=\sum_{\substack{t\equiv m\pmod{M}\\ p\nmid t}} t^{\kappa} H\left(4n-t^2\right).
\]
We are mostly interested in the case $n=p^r$ with $r\in\N$, which we abbreviate by $\mathscr{H}_{\kappa,m,M}(p^r):=\mathscr{H}_{\kappa,m,M}(p;p^r)$.  To state the result,  we set 
\begin{multline*}
E_{\kappa,m,M}\left(p^r\right):=\delta_{M\mid m}\delta_{\kappa=0}\delta_{2\nmid r} H(4p)+ \delta_{M\mid m}\delta_{\kappa=0}\delta_{2\mid r} \frac{1}{2}\left(1-\left(\frac{-1}{p}\right)\right)\\
 + \frac{1}{3}\left(1-\left(\frac{-3}{p}\right)\right)p^{\frac{r\kappa}{2}}\varrho_{\kappa,m,M}(p^r)+ \frac 13 (p-1) 2^{\kappa - 2}p^{\frac{r\kappa}{2}}\sigma_{\kappa,m,M}(p^r)
\end{multline*}
with 
\begin{align*}
\varrho_{\kappa,m,M}(p^r)&:=\displaystyle\sum_{\substack{t\equiv m\pmod{M}\\ t^2=p^r}}\sgn(t)^{\kappa},&\sigma_{\kappa,m,M}(p^r)&:=\displaystyle\sum_{\substack{t\equiv m\pmod{M}\\ t^2=4p^r}}\sgn(t)^{\kappa}.
\end{align*}
Here and throughout $\delta_{\mathcal{S}}:=1$
 if a statement $\mathcal{S}$ is true and $\delta_{\mathcal{S}}:=0$ otherwise.
We note that if $\kappa\in 2\N_0$, 
then $\varrho_{\kappa,m,M}(p^r)=\varrho_{m,M}(p^r)$, where 
\begin{align*}\label{eqn:varrhodef}
\varrho_{m,M}\left(p^r\right)&:=\#\left\lbrace t=\pm 2 p^{\frac{r}{2}}\in\Z:  t\equiv m\pmod{M}\right\rbrace,\\
\sigma_{m,M}\left(p^r\right)&:=\#\left\lbrace t=\pm p^{\frac{r}{2}}\in\Z: t\equiv m\pmod{M}\right\rbrace.
\end{align*}
\begin{lemma}\label{lem:S=Hp+E}
For a prime $p>3$, $\kappa\in\N_0$, and $r\in\N$ we have 
\begin{equation*}
2S_{\kappa,m,M}\left(p^r\right)=\mathscr{H}_{\kappa,m,M}\left(p^r\right)+E_{\kappa,m,M}(p^r).
\end{equation*}
\end{lemma}
\begin{proof}
The claim easily follows, using that by \cite[Theorem 3]{KaplanPetrow}, for a prime $p>3$ and $r\in\N$ we have 
\begin{equation*}%\label{eqn:NAt}
\hspace{100pt}
2N_{A}\left(p^r;t\right)=
\begin{cases}
H\left(4p^r-t^2\right)&\text{if }t^2<4p^r,\ p\nmid t,\\
H(4p)&\text{if }t=0\text{ and $r$ is odd},\\ \vspace{2pt}
\frac{1}{2}\left(1-\left(\frac{-1}{p}\right)\right)&\text{if }t=0\text{ and $r$ is even},\\ \vspace{0.5pt}
\frac{1}{3}\left(1-\left(\frac{-3}{p}\right)\right)&\text{if }t^2=p^r,\\
\frac{1}{12}\left(p-1\right)&\text{if }t^2=4p^r,\\ 
0&\text{otherwise.} \hspace{133pt} \qedhere  
\end{cases} %\vspace{-14pt}
\end{equation*} 
\end{proof}
The sums $\mathscr{H}_{\kappa,m,M}(p;n)$ are related to $H_{\kappa,m,M}(n)$, as a direct calculation shows. 
\begin{lemma}\label{lem:HpHrelGeneral}
For $m\in\Z$, $M\in\N$, $p$ prime, and $\kappa\in\N_0$,  we have
\[
\mathscr{H}_{\kappa,m,M}(p;n) = \sum_{\substack{\ell\pmod{p}\\ p\nmid (m+M\ell)}}H_{\kappa,m+M\ell,Mp}(n).
\]
\end{lemma}

\subsection{Generating functions for sums of moments of class numbers}
Taking the generating function of \eqref{eqn:HmMkdef}, for $m\in\Z$, $M\in\N$, and $\kappa\in\N_0$, we study sums of the type 
\begin{equation*}
\mathcal{H}_{\kappa,m,M}(\tau):= \sum_{n = 0}^{\infty} H_{\kappa,m,M}(n)q^n=\sum_{n = 0}^{\infty} \sum_{t \equiv m \pmod{M}} t^{\kappa} H\left(4n-t^2\right)q^n.
\end{equation*}
We directly see that 
\[
\mathcal{H}_{\kappa, m,M}=\left(\mathcal{H}\theta_{\kappa,m,M}\right)\big| U_4,
\]
where
\begin{equation*}
\theta_{\kappa,m,M}(\tau):=\sum_{n \equiv m \pmod M} n^{\kappa} q^{n^2}.
\end{equation*}

\subsection{Properties of class numbers}

We use the following bounds that follow from results of Siegel \cite{Siegel}.
\begin{lemma}\label{lem:Siegel}
For a discriminant $D<0$ and $\varepsilon>0$, we have 
\[
|D|^{\frac{1}{2}-\varepsilon}\ll_{\varepsilon} H(|D|)\ll_{\varepsilon} |D|^{\frac{1}{2}+\varepsilon}.
\]
\end{lemma}

If $D=\Delta f^2$ with $\Delta<0$ a fundamental discriminant and $f\in\N$, then (see \cite[p. 273]{Cohen})
\begin{equation}\label{eqn:CohenHrel}
H\left(|\Delta| f^2\right)= H(|\Delta|)\sum_{d\mid f}\mu(d)\left(\frac{\Delta}{d}\right)\sigma\left(\frac{f}{d}\right),
\end{equation}
where $\sigma(n):=\sum_{d\mid n} d$.
Setting $D_p:=\frac{D}{p^{2\alpha}}$ for an odd prime $p$ such that $2\alpha\leq \ord_p(D)\leq 2\alpha+1$, this leads to the following useful relation. 
\begin{lemma}\label{lem:HDHp^2relCohen}
If $D<0$ is a discriminant and $p$ is an odd prime, then
\[
H\left(|D|p^2\right)=pH(|D|) + \left(1-\left(\frac{D_p}{p}\right)\right)H\left(D_p\right).
\]
\end{lemma}
\begin{proof}
Let $D=\Delta f^2$ with $\Delta<0$ a fundamental discriminant and $f\in\N$. Since the sum in \eqref{eqn:CohenHrel} is multiplicative, if $f=p^{\alpha}g$ with $\alpha \in \N_0$ and $p\nmid g$, then we can write 
\begin{align}
\label{eqn:HDHDelta}
H\left(|D|\right)&=H\left(|\Delta|g^2\right)\left(\sigma\left(p^{\alpha}\right)-\delta_{\alpha\geq 1}\left(\frac{\Delta}{p}\right) \sigma\left(p^{\alpha-1}\right)\right),\\
\label{eqn:HDp^2HDelta} H\left(|D|p^2\right)&= H\left(|\Delta|g^2 \right)\left(\sigma\left(p^{\alpha+1}\right)-\left(\frac{\Delta}{p}\right) \sigma\left(p^{\alpha}\right)\right).
\end{align}
Note that for $\ell\geq 0$ we have
\[
\sigma\left(p^{\ell}\right)=1+\delta_{\ell\geq 1}p \sigma\left(p^{\ell-1}\right).
\]
Comparing \eqref{eqn:HDp^2HDelta} with \eqref{eqn:HDHDelta} yields 
\[
H\left(|D|p^2\right) =pH(|D|)+ \left(1-\left(\frac{\Delta}{p}\right)\right)H\left(|\Delta|g^2\right).
\]
Since $p\nmid g$, we have $(\tfrac{g^2}{p})=1$ and hence $(\tfrac{\Delta}{p}) = (\tfrac{\Delta g^2}{p})$. The claim follows from the fact that $\Delta g^2 = D_p$.
 \qedhere
\end{proof}

\subsection{Holomorphic projection}
 Let  $F(\tau) = \sum_{n \in \Z} c_{F,v}(n) q^{n}$ be a (not necessarily holomorphic) function satisfying weight $\kappa\geq 2$ modularity. Suppose furthermore that  $F(\tau)-P_{i\infty}(q^{-1})$ has moderate growth, where $P_{i\infty} \in \C [x]$ and that a similar condition holds as $\tau\to \Q$. We define (see \cite[Proposition 5.1, p. 288]{GrossZagier} for the general statement and \cite{MOR} for it written in this generality) the \begin{it}holomorphic projection\end{it} of $F$
\begin{align*}
\pi_{\text{hol}}^{\text{reg}} (F) (\tau) := P_{i \infty} \left( q^{-1} \right) + \sum_{n =1}^\infty c_{F}(n) q^{n}.
\end{align*}
Here for $n \in \N$
\begin{align*}
c_{F}(n) := \frac{(4 \pi n)^{\kappa-1}}{\Gamma(\kappa-1)} \lim_{s\to 0} \int_0^{\infty} c_{F,v} (n) v ^{\kappa-2-s} e^{-4 \pi n v} d v.
\end{align*}
Note that for a non-holomorphic modular form $F$, there exists a unique cusp form $f$ satisfying $\langle F,g\rangle=\langle f,g\rangle$ for every cusp form $g$, where $\langle \cdot , \cdot \rangle$ denotes the Petersson inner product. 
The function $f$ is Sturm's \cite{Sturm} original definition for the holomorphic projection of $F$ and Gross and Zagier showed in \cite[Proposition 5.1, p.288]{GrossZagier} that Sturm's definition matches the definition given here if one additionally assumes that $F$ decays polynomially towards all cusps.
We have the following properties of $\pi_{\operatorname{hol}}^{\operatorname{reg}}(F)$ (see \cite[Proposition 5.1 and Proposition 6.2]{GrossZagier} as well as \cite[(4.6)]{Me}). 
\begin{lemma}\label{lem:holproj}
Suppose that $F$ is continuous and transforms modular of weight $\kappa\geq 2$ on $\Gamma_1(N)$. Then the following hold. 
\begin{enumerate}[leftmargin=*,label={\rm(\arabic*)}]
\item If $F$ is holomorphic, then $\pi_{\operatorname{hol}}^{\operatorname{reg}}(F)=F$.
\item If $\kappa>2$ 
and $F$ is bounded towards all cusps, then $\pi_{\operatorname{hol}}^{\operatorname{reg}}(F)$ is a holomorphic modular form. If $\kappa=2$, then $\pi_{\operatorname{hol}}^{\operatorname{reg}}(F)$ is a 
quasimodular form of weight two.
\item 
If $F_1$ is a weight $\kappa_1 \in\frac{1}{2}\N$ harmonic Maass form and  $f_2$ is a weight $\kappa_2\in\frac{1}{2}\N$ holomorphic modular form, and $k\in\N_0$ with $\kappa:=\kappa_1+\kappa_2+2k\geq 2$, then 
\[
\pi_{\operatorname{hol}}^{\operatorname{reg}}\left(\left[F_1,f_2\right]_{k}\right) = \left[F_1^+,f_2\right]_k +\pi_{\operatorname{hol}}^{\operatorname{reg}}\left(\left[F_1^-,f_2\right]_{k}\right).
\]
\end{enumerate}
\end{lemma}
\begin{remark}
Lemma \ref{lem:holproj} (3) appears in a slightly different form in \cite[(4.6)]{Me} because Mertens did not include the constant term in the definition of $F^-$ in \cite[(3.1)]{Me}. 
\end{remark}
Rearranging Lemma \ref{lem:holproj} (3) yields a formula for $[F_1^+,f_2]_{k}$. In \cite[Theorem 1.2]{Me}, Mertens considered the special case that $\kappa_1=\frac{3}{2}$, $\kappa_2=\frac{1}{2}$, and $\xi_{\frac{3}{2}}(F_1)$ and $F_2$ are both weight $\frac{1}{2}$ holomorphic modular forms on $\Gamma_1(N)$.  Serre and Stark showed in \cite[Theorem A]{SerreStark} that this space is spanned by unary theta functions. For a character $\chi$ and $a\in\N$,  these are given by $\theta_{\chi}\mid V_a$, where
\[
\theta_{\chi}(\tau):=\sum_{n\in\Z} \chi(n) q^{n^2}.
\]
Using Serre and Stark's classification, one may assume without loss of generality that $F_2=\theta_{\chi}\big|V_a$ and $\xi_{\frac{3}{2}}(F_1)=\theta_{\psi}\big|V_b$ for some characters $\chi,\psi$ and $a,b\in\N$.  In the proof of \cite[Theorem 1.2]{Me}, Mertens related the second term on the right-hand side of Lemma \ref{lem:holproj} (3) to 
\begin{equation*}%\label{Lambdastdef}
\Lambda_{\ell,a,b}^{\chi,\psi}(\tau):=2\sum_{n=1}^{\infty}\lambda_{\ell,a,b}^{\chi,\psi}(n)q^n,
 \qquad \text{where}\qquad
\lambda_{\ell,a,b}^{\chi,\psi}(n):=\sideset{}{^*}\sum_{\substack{at^2-bs^2=n\\ t\geq 1,s\geq 0}}\chi(t)\overline{\psi(s)}\left(t\sqrt{a}-s\sqrt{b}\right)^{\ell},
\end{equation*}
where here and throughout {\scriptsize$\sideset{}{^*}\sum$}  means that the terms in the sum with $s=0$ are weighted by $\frac{1}{2}$. 
Using the fact that $\pi_{\operatorname{hol}}^{\operatorname{reg}}([F_1,F_2]_{k})$ is a quasimodular form by Lemma \ref{lem:holproj} (2), one then obtains the following.
\begin{lemma}\label{lem:Mertens}
Suppose that $\chi$ and $\psi$ are characters of conductors $N_{\chi}$ and $N_{\psi}$, respectively and $N,a,b\in\N$ with $bN_{\psi}^2\mid N$. If $F$ is a harmonic Maass form of weight $\frac{3}{2}$ on $\Gamma_1(4N)$ that grows at most polynomially towards all cusps and satisfies  $\xi_{\frac{3}{2}}(F)=\theta_{\psi}\big|V_b$, then 
\[
\left(\left[F^+,\theta_{\chi}\big|V_a\right]_{k} -2^{3-2k}\pi\binom{2k}{k}  \Lambda_{2k+1,a,b}^{\chi,\psi}\right)\Big|U_4
\]
is a holomorphic cusp form of weight $2k+2$ on $\Gamma_1(\operatorname{lcm}(4N,4aN_{\chi}^2))$ if $k>0$ and a quasimodular form of weight two if $k=0$. 
\end{lemma}
\begin{remark}
The statement of Lemma \ref{lem:Mertens} corrects an error in \cite[(5.2)]{Me} where the constant in front of the second term differs by a factor of $-2\sqrt{\pi}$.
\end{remark}

\section{Holomorphic projection and the proof of Theorem \ref{thm:HmMk}}\label{sec:holprojH}
 In this section, we prove Theorem \ref{thm:HmMk}. 
\subsection{Fourier coefficients of certain Rankin--Cohen brackets}

 Define (compare with \cite[(7.7)]{Me}, although the notation is different there) 
\begin{equation*}%\label{G}
G_{k,m,M}(n) := \sum_{\substack{t\in\Z\\ t\equiv m\pmod{M}}}p_{2k}(t,n) H\left(4n-t^2\right),
\end{equation*}
where $p_{2k}(t,n)$ 
 denotes the $(2k)$-th coefficients in the Taylor expansion of $(1-tX+nX^2)^{-1}$. 
These appear in the Fourier expansion of $[\mathcal{H}, \theta_{m,M}]_{k} |U_4$ as a direct calculation using the results of Cohen \cite{Cohen} shows. 
\begin{lemma}\label{lem:GcoeffRankinCohen}
The $n$-th Fourier coefficient of $[\mathcal{H}, \theta_{m,M}]_{k} |U_4$ equals \(\frac{(2k)!}{2\cdot k!} G_{k,m,M}(n)\). 
\end{lemma}
Using the explicit evaluation (\cite[(3), page 29]{EichlerZagier})
\begin{equation*}%\label{ptn}
p_{2k}(t,n)= \frac{(2k)!}{k!}\sum_{\mu=0}^k (-1)^\mu  \binom{2k-\mu}{\mu}
t^{2 k-2 \mu}n^\mu
\end{equation*}
we directly obtain the following lemma.
\begin{lemma}\label{lem:HrelateG}
For $m\in\Z$ and $k,M\in\N$, we have 
\begin{equation*}
H_{2k,m,M}(n) = \frac{k!}{(2k)!} G_{k,m,M}(n) - \sum_{\mu=1}^{k} (-1)^\mu \frac{(2k-\mu)!}{\mu!(2k-2\mu)!} n^\mu H_{2k-2\mu,m,M}(n).
\end{equation*}
\end{lemma}

In order to investigate $H_{2k,m,M}(n)$, it suffices by Lemma \ref{lem:GcoeffRankinCohen} and Lemma \ref{lem:HrelateG}  to study the coefficients of $[\mathcal{H},\theta_{m,M}]_k|U_4$. In \cite[Theorem 1.2]{Me} Mertens applied holomorphic projection on functions related to $[\widehat{\mathcal{H}},\theta_{m,M}]_k|U_4$. To state his result (noted two paragraphs before \cite[Proposition 7.2]{Me}), we set 
\begin{align*}%\label{eqn:Lambdadef}
\Lambda_{\ell,m,M}(\tau) &:= \sum_{n=1}^{\infty} \lambda_{\ell,m,M}(n) q^n,\qquad \text{where} \qquad
\lambda_{\ell,m,M}(n):=\sum_{\pm} \sideset{}{^*}\sum_{\substack{t>s\geq 0\\ t^2-s^2=n \\ t\equiv \pm m\pmod{M}}}(t-s)^{\ell}.
\end{align*}
 \begin{lemma}\label{lem:Mertensholproj}
For $k\in\N_0$, $m\in\Z$, and $M\in\N$, the function 
\begin{equation*}
\left(\left[\mathcal{H},\theta_{m,M}\right]_k +2^{-1-2k}\binom{2k}{k}\Lambda_{2k+1,m,M}\right)\bigg|U_4
\end{equation*}
is a holomorphic cusp form of weight $2+2k$ on $\Gamma_0(4M^2)\cap\Gamma_1(M)$ (resp. $\Gamma_0(4M^2)$) if $M\nmid m$ (resp. $M\mid m$) if $k>0$ and quasimodular on that group if $k=0$.
\end{lemma}
\begin{remark}
Lemma \ref{lem:Mertensholproj} was stated in a different form in \cite{Me}.  Specifically, Mertens assumed that $M$ is prime and an error appears in the constant in front of $\Lambda_{2k+1,m,M}$ in \cite{Me}. The proof in \cite{Me} goes through without the assumption that $M$ is prime, however.
\end{remark}
\begin{proof}[Proof of Lemma \ref{lem:Mertensholproj}]
Using Lemma \ref{lem:Mertens} and relating the functions $\Lambda_{\ell,a,b}^{\chi,\psi}$ (resp. $[\mathcal{H},\theta_{\chi}|V_{a}]_{k}$) to $\Lambda_{\ell,m,M}$ (resp. $[\mathcal{H},\theta_{m,M}]_{k}$) via orthogonality of characters yields the claim.
\end{proof}

In light of Lemma \ref{lem:HrelateG}, it is natural to recursively define $\mathscr{C}_0:=1$ and 
\[
\mathscr{C}_k:=-\sum_{\mu=1}^{k} (-1)^\mu \frac{(2k-\mu)!}{\mu!(2k-2\mu)!} \mathscr{C}_{k-\mu}.
\]
As we show in the next lemma, $\mathscr{C}_k$ are the \begin{it}Catalan numbers\end{it} 
\[
C_{k}:=\frac{1}{k+1} \binom{2k}{k}.
\]
\begin{lemma}\label{lem:Catalan}
We have $\mathscr{C}_k=C_k$.
\end{lemma}
\begin{proof}
For $k=0$ the claim holds directly by definition. Assume inductively that the claim holds for all $\ell<k$. Then we have 
\[
\mathscr{C}_{k}=-\sum_{\mu=1}^{k} (-1)^\mu \binom{2k-\mu}{\mu}C_{k-\mu}.
\]
It remains to show that 
\begin{equation}\label{eqn:CatalanRecursive}
-\sum_{\mu=1}^{k} (-1)^\mu \binom{2k-\mu}{\mu}C_{k-\mu}=\begin{cases} C_k&\text{if }k\geq 1,\\ 0&\text{if }k=0.\end{cases}
\end{equation}
To show the claim, we take the generating function  of the left-hand side of \eqref{eqn:CatalanRecursive}: 
\begin{align}
L(X):&=-\sum_{k = 1}^\infty \sum_{\mu=1}^k (-1)^{\mu} \binom{2k-\mu}{\mu} C_{k-\mu}X^{k}
\label{eqn:almostCatalan} \\
&=-\frac{1}{1+X} \sum_{k = 0}^\infty C_k \left(\frac{X}{(1+X)^2}\right)^k+\sum_{k=0}^\infty C_k X^{k}, \nonumber
\end{align}
using the binomial series expansion. We then recall the evaluation of the generating function (see \cite[26.5.2]{NIST})
\[
F(Z):=\sum_{k = 0}^\infty C_k Z^k= \frac{ 1- \sqrt{ 1-4Z}}{2Z},
\]
valid for $|Z|<\frac{1}{4}$. Therefore, for $0<X<\frac{1}{4}$, we obtain
\begin{align*}
F\left(\frac{X}{(1+X)^2}\right)=1+X.
\end{align*}
Plugging this into the first sum in \eqref{eqn:almostCatalan} yields $L(X)=\sum_{k = 1}^\infty C_k X^{k}$. This gives the claim.
\end{proof}

\subsection{Asymptotic growth of the moments}
In order to obtain the asymptotic growth of $G_{k,m,M}$ and $H_{2k,m,M}$, we require the following straightforward estimates.
\begin{lemma}\label{lem:lambdabnd}
We have
\[
\lambda_{\ell,m,M}(n)\leq n^{\frac{\ell}{2}}\lambda_{m,M}(n)\ll_{\varepsilon} n^{\frac{\ell}{2}+\varepsilon}.
\]
\end{lemma}
We are now ready to prove an asymptotic formula for $H_{2k,m,M}(n)$.
\begin{proposition}\label{prop:HmMkbound}
We have
\[
H_{2k,m,M}(n)=C_k n^{k}H_{m,M}(n) + O_{k,M,\varepsilon}\left(n^{k+\frac{1}{2}+\varepsilon}\right).
\]
\end{proposition}
\begin{proof}
We argue by induction using Lemma \ref{lem:HrelateG}.  Since $C_0=1$, the claim holds trivially for $k=0$.  For $k\geq1$, Lemma \ref{lem:GcoeffRankinCohen} and Lemma \ref{lem:Mertensholproj} imply that
\[
G_{k,m,M}(n)+\frac{1}{2^{2k}\cdot k!}\lambda_{2k+1,m,M}(4n)
\]
is the $n$-th coefficient of a weight $2k+2$ cusp form. By Deligne's bound \cite{Deligne} it thus may be bound against $O_{k,M,\varepsilon}(n^{k+\frac{1}{2}+\varepsilon})$. 
The implied constant in the error term a priori depends on $m$ as well,  but by taking the maximum over all of the choices of $m\pmod{M}$, we may drop the depencence on $m$ throughout. Using Lemma \ref{lem:lambdabnd}, we obtain
\begin{equation*}%\label{eqn:Gbound}
G_{k,m,M}(n)\ll_{k,M,\varepsilon}  n^{k+\frac{1}{2}+\varepsilon}.
\end{equation*}
Plugging this into Lemma \ref{lem:HrelateG}, using the inductive hypothesis 
and the fact that $C_k$ satisfies the recurrence defining $\mathscr{C}_k$ 
by Lemma \ref{lem:Catalan}, we obtain the claim.
\end{proof}

\subsection{The main term}
In this subsection we investigate the growth of $H_{m,M}(n)$. 
\begin{lemma}\label{lem:HmM0lower}
For $m,M\in\N$ fixed, we have, as $n\to\infty$, 
\[
n^{1-\varepsilon}\ll_{\varepsilon,M} H_{m,M}(n)\ll_{\varepsilon} n^{1+\varepsilon}.
\]
\end{lemma}
\begin{proof}
We begin with the lower bound.  Since for $4n-t^2\neq 0$, $H(4n-t^2)\geq 0$ and since $H(0)=-\frac{1}{12}$, we obtain
\begin{equation}\label{eqn:HmMpartial}
H_{m,M}(n)\geq \sum_{\substack{t\equiv m\pmod{M}\\ |t|\leq \sqrt{n}}} H\left(4n-t^2\right)-\frac{1}{6}.
\end{equation}
Using the lower bound in Lemma \ref{lem:Siegel}, we obtain 
\[
\sum_{\substack{t\equiv m\pmod{M}\\ t\leq \sqrt{n}}} H\left(4n-t^2\right)\gg_{\varepsilon}\sum_{\substack{t\equiv m\pmod{M}\\ t\leq \sqrt{n}}} \left(4n-t^2\right)^{\frac{1}{2}-\varepsilon}\gg_{\varepsilon,M} n^{1-\varepsilon}.
\]
Plugging this into \eqref{eqn:HmMpartial} yields the lower bound.  

For the upper bound, we again use the fact that every term is non-negative except $4n-t^2=0$ to bound 
\[
H_{m,M}(n)\leq \sum_{|t|< 2\sqrt{n}} H\left(4n-t^2\right).
\]
The upper bound in Lemma \ref{lem:Siegel} then yields 
\[
\sum_{|t|< 2\sqrt{n}} H\left(4n-t^2\right)\ll_{\varepsilon}\sum_{|t|< 2\sqrt{n}} \left(4n-t^2\right)^{\frac{1}{2}+\varepsilon}\ll_{\varepsilon} \sum_{|t|< 2\sqrt{n}} n^{\frac{1}{2}+\varepsilon}\ll n^{1+\varepsilon}.\qedhere
\]
\end{proof}

We are now ready to prove Theorem \ref{thm:HmMk}.
\begin{proof}[Proof of Theorem \ref{thm:HmMk}]
By Proposition \ref{prop:HmMkbound}, we have
\[
\frac{H_{2k,m,M}(n)}{n^{k}H_{m,M}(n)}=C_k  + O_{k,M,\varepsilon}\left(\frac{n^{\frac{1}{2}+\varepsilon}}{H_{m,M}(n)}\right).
\]
The claim now follows by using the lower bound in Lemma \ref{lem:HmM0lower} to bound the error term. 
\end{proof}

\section{Proof of Theorem \ref{thm:SmMk}}\label{sec:EllCurveApplication}
The next lemma relates $\mathscr{H}_{2k,m,M}(p^r)$ to linear combinations of  $H_{2k,m,M}(p^{j})$ with $0\leq j\leq r$.

\begin{lemma}\label{lem:HpHrel}
Suppose that $p>3$, $m\in\Z$, $M\in\N$, and $k\in\N_0$.
 \noindent

\noindent
\begin{enumerate}[leftmargin=*,label={\rm(\arabic*)}]
\item If $r\leq 1$ or both $p\mid M$ and $p\nmid m$,  then 
\[
\mathscr{H}_{2k,m,M}\left(p^r\right)=H_{2k,m,M}\left(p^r\right)-\delta_{M\mid m}\delta_{k=0} H(4p).
\]
\item
If $p\nmid M$ and $r\geq 2$, then 
\begin{multline*}
\mathscr{H}_{2k,m,M}\left(p^r\right)=H_{2k,m,M}\left(p^r\right)-p^{2k+1} H_{2k,m\overline{p},M}\left(p^{r-2}\right)-\delta_{M\mid m}\delta_{k=0}\delta_{2\mid r}\frac{1}{2}\left(1-\left(\frac{-1}{p}\right)\right)\\
 - \delta_{M\mid m}\delta_{k=0}\delta_{2\nmid r}H(4p)- \tfrac{4^{k}}{12} p^{rk+1}\left(1-\frac{1}{p}\right)\varrho_{m,M}(p^{r})-\frac{1}{3}\left(1-\left(\frac{-3}{p}\right)\right)p^{rk}\sigma_{m,M}\left(p^r\right).
\end{multline*}
\end{enumerate}
\end{lemma}
\begin{proof}
(1) By definition, we have 
\begin{align}\label{eqn:HpHincexc}
\mathscr{H}_{2k,m,M}(p^r)= H_{2k,m,M}(p^r)-\sum_{\substack{t\equiv m\pmod{M}\\ p\mid t}}  t^{2k} H\left(4p^r-t^2\right).
\end{align}
It is not hard to see that under the assumptions of (1) the only possible term in the second sum is $t=0$, which only occurs if $M\mid m$ and $k=0$, giving the claim.

\noindent
(2)  Letting $t\mapsto pt$, the second summand in \eqref{eqn:HpHincexc} equals (we assume that $r\geq 2$)
\begin{equation}\label{eqn:momentdivbyp}
p^{2k}\sum_{\substack{pt\equiv m\pmod{M}}} t^{2k} H \left(\left(4p^{r-2}-t^2\right)p^2\right)
=p^{2k}\sum_{\substack{t\equiv m\overline{p}\pmod{M}}} t^{2k} H \left(\left(4p^{r-2}-t^2\right)p^2\right).
\end{equation}
We next use Lemma \ref{lem:HDHp^2relCohen} with $D=t^2-4p^{r-2}$ to rewrite the terms in \eqref{eqn:momentdivbyp} with $|t|<2p^{\frac{r}{2}-1}$, where this condition is required to assure that $D<0$. If $0<|t|<2p^{\frac{r}{2}-1}$ and $p>3$,  then $\ord_p(t^2-4p^{r-2})=\ord_p(t^2)$ and hence $\alpha=\ord_p(t)$ in Lemma \ref{lem:HDHp^2relCohen}, yielding
\begin{multline}\label{eqn:scrHrewrite}
H\left(\left(4p^{r-2}-t^2\right)p^2\right)=p H\left(4p^{r-2}-t^2\right)\\
+\left(1-\left(\frac{\left(\frac{t}{p^{\alpha}}\right)^2-4p^{r-2-2\alpha}}{p}\right)\right)H\left(4p^{r-2-2\alpha}-\left(\frac{t}{p^{\alpha}}\right)^2\right).
\end{multline}
For $t=0$, the choice of $\alpha$ in Lemma \ref{lem:HDHp^2relCohen} is 
$\alpha=\frac{r}{2}-1$ if $r$ is even and $\alpha=\frac{r-3}{2}$ if $r$ is odd, 
so in this case Lemma \ref{lem:HDHp^2relCohen} gives
\[
H\left(4p^r\right) = pH\left(4p^{r-2}\right)+ \delta_{2\mid r}\left(1-\left(\frac{-1}{p}\right)\right)H(4) + \delta_{2\nmid r} H(4p).
\]
Noting that the Legendre symbol in \eqref{eqn:scrHrewrite} equals $1$ unless $\ord_p(t)=\frac{r}{2}-1$, plugging back into \eqref{eqn:momentdivbyp} for the $|t|<2p^{\frac{r}{2}-1}$ terms and using the evaluations $H(0)=-\frac{1}{12}$, $H(3)=\frac{1}{3}$,  and $H(4)=\frac{1}{2}$ easily yields the claim.
\end{proof}

We next bound $|E_{2k,m,M}(p^r)|$. 
\begin{lemma}\label{lem:Ebound}
For $k\in\N_0$ and $\varepsilon > 0$ we have 
\[
E_{2k,m,M}\left(p^r\right)=\frac 13 2^{2k-2}p^{rk+1} \varrho_{m,M}\left(p^r\right)  +O_{k,\varepsilon}\left(p^{rk+\frac{\delta_{k=0}}{2}+\varepsilon}\right).
\]
\end{lemma}
\begin{proof} 
The first term in the definition of $E_{2k,m,M}(p^r)$ only occurs if 
 $k=0$ and we use the upper bound in Lemma \ref{lem:Siegel} to conclude that 
\[
H(4p)\ll_{\varepsilon} p^{\frac{1}{2}+\varepsilon}=p^{rk+\frac{\delta_{k=0}}{2}+\varepsilon}.
\]
The second term is clearly $O(1)$ and the third term is $O(p^{rk})$ because $0\leq \varrho_{m,M}(p^r)\leq 2$. We finally use $0\leq \sigma_{m,M}(p^r)\leq 2$ to split the last term in the definition of $E_{2k,m,M}(p^r)$ as 
\[
\frac 13 \left(p-1\right) 2^{2k-2}p^{rk}\sigma_{m,M}\left(p^r\right) = \frac 13 p^{2k-2}p^{rk}\sigma_{m,M}\left(p^r\right) +O_{k}\left(p^{rk}\right),
\]
yielding the claim. 
\end{proof}
We are now ready to prove Theorem \ref{thm:SmMk}.
\begin{proof}[Proof of Theorem \ref{thm:SmMk}]
\noindent
(1) By Lemma \ref{lem:S=Hp+E}, we have 
\begin{equation}\label{eqn:SmMkratio}
\frac{S_{2k,m,M}\left(p^r\right)}{p^{rk}S_{m,M}\left(p^r\right)} = \frac{\mathscr{H}_{2k,m,M}\left(p^r\right)+E_{2k,m,M}\left(p^r\right)}{p^{rk}\left(\mathscr{H}_{m,M}\left(p^r\right)+E_{m,M}\left(p^r\right)\right)}.
\end{equation}
We first consider $r=1$. By Lemma \ref{lem:HpHrel} (1) the right-hand side of \eqref{eqn:SmMkratio} equals 
\begin{equation*}%\label{eqn:SmMkratio2}
\frac{H_{2k,m,M}(p)+E_{2k,m,M}(p)}{p^{k}\left(H_{m,M}(p)-\delta_{M\mid m}\delta_{k=0}H(4p)+E_{m,M}(p)\right)}=\frac{\frac{H_{2k,m,M}(p)}{p^{k}H_{m,M}(p)} + \frac{E_{2k,m,M}(p)}{p^{k}H_{m,M}(p)}}{1-\delta_{M\mid m}\delta_{k=0}\frac{H(4p)}{H_{m,M}(p)}+ \frac{E_{m,M}(p)}{H_{m,M}(p)}}.
\end{equation*}
Using Lemma \ref{lem:Ebound} and Lemma \ref{lem:HmM0lower} together with Lemma \ref{lem:Siegel}, we obtain (note that $\varrho_{m,M}(p^r)=0$ for $r$ odd)
\begin{align*}
\frac{H(4p)}{H_{m,M}(p)}\ll_{M,\varepsilon} p^{-\frac{1}{2}+\varepsilon}, \quad
\frac{|E_{2k,m,M}(p)|}{p^{k}H_{m,M}(p)}\ll_{k,M,\varepsilon} p^{\frac{\delta_{k=0}}{2}-1+\varepsilon}\ll p^{-\frac{1}{2}+\varepsilon}.
\end{align*}
Using Theorem \ref{thm:HmMk} yields the case $r=1$.

For $r\geq 2$, we use Lemma \ref{lem:HpHrel} (2) to rewrite 
\begin{equation}\label{eqn:H2pHrel1}
\mathscr{H}_{2k,m,M}\left(p^r\right) = H_{2k,m,M}\left(p^r\right) - p^{2k+1}H_{2k,m\overline{p},M}\left(p^{r-2}\right)+O_{k,M}\left(p^{rk+1}\right).
\end{equation}
By Proposition \ref{prop:HmMkbound} and the upper bound in Lemma \ref{lem:HmM0lower}, we have 
\begin{align*}
 p^{2k+1}H_{2k,m\overline{p},M}\left(p^{r-2}\right)&= C_k p^{rk+1}H_{m\overline{p},M}\left(p^{r-2}\right) +O_{k,M,\varepsilon}\left(p^{2k+1+ (r-2)\left(k+\frac{1}{2}+\varepsilon\right)}\right)\\
&\ll_{k,M,\varepsilon} p^{r(k+1+\varepsilon) -1}+p^{r\left(k+\frac{1}{2}+\varepsilon\right)}\ll_{\varepsilon} p^{r(k+1+\varepsilon) -1}, 
\end{align*}
where in the last bound we use the fact that $r-1\geq \frac{r}{2}$ for $r\geq 2$. Plugging this back into \eqref{eqn:H2pHrel1} yields
\begin{equation}\label{eqn:Hpprbound}
\mathscr{H}_{2k,m,M}\left(p^r\right) = H_{2k,m,M}\left(p^r\right) +O_{k,M,\varepsilon}\left(p^{r(k+1+\varepsilon)-1}\right).
\end{equation}
By Lemma \ref{lem:Ebound}, we have
\begin{equation}\label{eqn:Eprbound}
E_{2k,m,M}\left(p^r\right) = O_{k}\left(p^{rk+1}\right)= O_{k}\left(p^{r(k+1+\varepsilon)-1}\right).
\end{equation}
Plugging \eqref{eqn:Hpprbound} and \eqref{eqn:Eprbound} into \eqref{eqn:SmMkratio} yields
\begin{align*}
\frac{S_{2k,m,M}\left(p^r\right)}{p^{rk}S_{m,M}\left(p^2\right)} &= \frac{H_{2k,m,M}\left(p^r\right)+O_{k,M,\varepsilon}\left(p^{r(k+1+\varepsilon)-1}\right)}{p^{rk}H_{m,M}\left(p^r\right)+O_{k,M,\varepsilon}\left(p^{r(k+1+\varepsilon)-1}\right)}
= \frac{\frac{H_{2k,m,M}\left(p^r\right)}{p^{rk}H_{m,M}\left(p^r\right)} + O_{k,M,\varepsilon}\left(\frac{p^{r(1+\varepsilon)-1}}{H_{m,M}\left(p^r\right)}\right)}{ 1+ O_{k,M,\varepsilon}\left(\frac{p^{r(1+\varepsilon)-1}}{H_{m,M}\left(p^r\right)}\right)}.
\end{align*}
The lower bound in Lemma \ref{lem:HmM0lower} now yields the claim.

\noindent
(2)
First assume that $p\mid M$. Since $p\nmid m$ in this case by assumption, Lemma \ref{lem:HpHrel} (1) yields that $\mathscr{H}_{2k,m,M}(p^r)=H_{2k,m,M}(p^r)$ and by \eqref{eqn:SmMkratio} we have
\[
\frac{S_{2k,m,M}\left(p^r\right)}{p^{rk}S_{m,M}\left(p^r\right)} = \frac{H_{2k,m,M}\left(p^r\right)+E_{2k,m,M}\left(p^r\right)}{p^{rk}\left(H_{m,M}\left(p^r\right)+E_{m,M}\left(p^r\right)\right)}=\frac{\frac{H_{2k,m,M}\left(p^r\right)}{p^{rk}H_{m,M}\left(p^r\right)} + \frac{E_{2k,m,M}(p^r)}{p^{rk}H_{m,M}\left(p^r\right)}}{1+\frac{E_{m,M}(p^r)}{H_{m,M}\left(p^r\right)}}.
\]
By Lemma \ref{lem:Ebound} and Lemma \ref{lem:HmM0lower} we obtain 
\[
\frac{E_{2k,m,M}\left(p^r\right)}{p^{rk}H_{m,M}\left(p^r\right)}\ll_{p,M,\varepsilon}  p^{r\left(-\frac{1}{2}+\varepsilon\right)}
\]
and the proof follows as in the case $r=1$.  
 
Next suppose that $p\nmid M$. We first rewrite the numerator of \eqref{eqn:SmMkratio}. Plugging Theorem \ref{thm:HmMk} into the right-hand side of Lemma \ref{lem:HpHrelGeneral} yields 
\begin{equation}\label{eqn:HpHrelThmHmMk}
\mathscr{H}_{2k,m,M}\left(p^r\right)= \sum_{\substack{\ell\pmod{p}\\ p\nmid (m+M\ell )}} p^{rk} H_{m+M\ell,Mp}\left(p^r\right)\left(C_k+O_{k,p,M,\varepsilon}\left(p^{r\left(-\frac{1}{2}+\varepsilon\right)}\right)\right).
\end{equation}
We then insert Lemma \ref{lem:HpHrelGeneral} into the right-hand side of \eqref{eqn:HpHrelThmHmMk} to obtain
\[
\mathscr{H}_{2k,m,M}\left(p^r\right)= p^{rk}\mathscr{H}_{m,M}\left(p^r\right)\left( C_k+O_{k,p,M,\varepsilon}\left(p^{r\left(-\frac{1}{2}+\varepsilon\right)}\right)\right).
\]
Plugging back into \eqref{eqn:SmMkratio} yields 
\begin{equation}\label{eqn:SmMkratioBigrSum}
\frac{S_{2k,m,M}\left(p^r\right)}{p^{rk}S_{m,M}\left(p^r\right)} = \frac{ C_k+O_{k,p,M,\varepsilon}\left(p^{r\left(-\frac{1}{2}+\varepsilon\right)}\right) + O\left(\frac{E_{2k,m,M}\left(p^r\right)}{p^{rk}\mathscr{H}_{m,M}\left(p^r\right)}\right)}{1+ O\left(\frac{E_{m,M}\left(p^r\right)}{ \mathscr{H}_{m,M}\left(p^r\right)}\right)}.
\end{equation}
By Lemma \ref{lem:HpHrelGeneral} and Lemma \ref{lem:HmM0lower}, for any choice $\lambda\pmod{p}$ such that $p\nmid (m+M\lambda)$ 
\[
\mathscr{H}_{m,M}\left(p^r\right)= \sum_{\substack{\ell\pmod{p}\\ p\nmid (m+M\ell)}}H_{m+M\ell,Mp}(p^r)\geq H_{m+M \lambda,Mp}\left(p^r\right)\gg_{p,M,\varepsilon} p^{r(1-\varepsilon)}.
\]
Hence by Lemma \ref{lem:Ebound} we have
\[
\frac{E_{2k,m,M}\left(p^r\right)}{p^{rk}\mathscr{H}_{m,M}\left(p^r\right)}\ll_{k,p,M,\varepsilon} \frac{p^{rk+\frac{1}{2}+\varepsilon}}{p^{r(k+1-\varepsilon)}}\ll_{p,\varepsilon} p^{r\left(-1+\varepsilon\right)}.
\]
The claim now follows from \eqref{eqn:SmMkratioBigrSum}.
\end{proof}


\begin{thebibliography}{99}
\bibitem{AS} O.  Ahmadi and I.E. Shparlinski,  \begin{it}  On the distribution of the number of points on algebraic curves in extentions of finite fields \end{it},  arXiv:0907.3664, 2009.
\bibitem{Artin}  E. Artin, \begin{it}Quadratische K\"orper im Gebiete der h\"oheren Kongruenzen. II. Analytischer Teil,\end{it} Math. Z. \textbf{19} (1924), 207--246.
%\bibitem{AKN} T. Asai, M. Kaneko, and H. Ninomiya, \emph{Zeros of certain modular functions and an application}, Comm. Math. Univ. Sancti Pauli \textbf{46} (1997),  93--101.
\bibitem{BarbanGordover} M. B. Barban and G. Gordover, \begin{it}Moments of the number of classes of purely radical quadratic forms with negative determinant\end{it}, Sov. Math. \textbf{7} (1966), 356--358.
\bibitem{BGHT} T. Barnet-Lamb, D. Geraghty, M. Harris, and R. Taylor, \begin{it}A family of Calabi--Yau varieties and potential automorphy II\end{it}, Publ. Res. inst. Math. Sci. \textbf{47} (2011), 29--98.
%               \bibitem{Beardon} A. Beardon, \begin{it} The geometry of discrete groups\end{it}, Grad. Texts in Math. \textbf{91}, Springer, New York, 1995.
\bibitem{BhandMurty}A. Bhand and M. R. Murty, \begin{it}Class number of quadratic fields\end{it}, Hardy-Ramanujan J. \textbf{42} (2019), 17--25.
\bibitem{Birch} B. Birch, \begin{it}How the number of points of an elliptic curve over a fixed prime field varies\end{it}, J. London Math. Soc. \textbf{43} (1968), 57--60.
%\bibitem{Bo1}R. Borcherds, \begin{it}Automorphic forms with singularities on Grassmannians\end{it}, Invent. Math. \textbf{132} (1998), 491--562.

%\bibitem{BCD} C. Breuil, B. Conrad, F. Diamond, and R. Taylor, \begin{it} On the modularity of elliptic curves over $\Q$: wild 3-adic exercises\end{it}, J. Amer. Math. Soc. \textbf{14} (2001), 843--939.

%\bibitem{BringmannDeform}K. Bringmann, \begin{it}On the explicit construction of higher deformations of partition statistics\end{it}, Duke Math. J. \textbf{144} (2008), 195--233.
%\bibitem{BringmannTaylor}K. Bringmann, \begin{it}Taylor coefficients of non-holomorphic Jacobi forms and applications\end{it}, Res. Math. Sci. \textbf{5} (2018), 5--15.
%\bibitem{BDE} K. Bringmann, N. Diamantis, and S. Ehlen, \begin{it}Regularized inner products and Eichler cocycles\end{it}, Int. Math. Res. Not., to appear.
%\bibitem{BDR} K. Bringmann, N. Diamantis, and M. Raum, \begin{it}Mock period functions, sesquiharmonic Maass forms, and non-critical value of $L$-functions\end{it}, Adv. Math. \textbf{233} (2013), 115--134.
%\bibitem{BringmannFolsom}K. Bringmann and A. Folsom, \begin{it}Almost harmonic Maass forms and Kac--Wakimoto characters\end{it}, J. reine angew. Math. \textbf{2014} (2014), 179--202.
 %\bibitem{BGK}K. Bringmann, P. Guerzhoy, and B. Kane, \begin{it}Shintani lifts and fractional derivatives for harmonic weak Maass forms\end{it}, Adv. Math. \textbf{255} (2014), 641--671.
%\bibitem{BJangK}K. Bringmann, M. Jang, and B. Kane, \begin{it}Representations of integers as sums of four polygonal numbers and partial theta functions\end{it}, submitted for publication. 
%\bibitem{BJK}K. Bringmann, P. Jenkins, and B. Kane, \begin{it}Differential operators on polar harmonic Maass forms and elliptic duality\end{it}, Quart. J. Math. \textbf{70} (2019), 1181--1207.
%\bibitem{BKweight0}K. Bringmann and B. Kane, \begin{it}A problem of Petersson about weight 0 meromorphic modular forms\end{it}, Res. Math. Sci \textbf{3:24} (2016), 1--31.
\bibitem{BKClassNum}K. Bringmann and B. Kane, \begin{it}Sums of class numbers and mixed mock modular forms\end{it}, Math. Proc. Cambridge Phil. Soc. \textbf{167} (2019), 321--333.
%\bibitem{BKFC} K. Bringmann and B. Kane, \begin{it}Ramanujan and coefficients of meromorphic modular forms\end{it}, J. Math. Pures Appl. \textbf{107} (2017), 100--122.
%\bibitem{BKFC2}K. Bringmann and B. Kane, \begin{it}Ramanujan-like formulas for Fourier coefficients of all meromorphic cusp forms\end{it}, submitted for publication.
%\bibitem{BKLOR} K. Bringmann, B. Kane, S. L\"obrich, K. Ono, and L. Rolen, \begin{it}On divisors of modular forms\end{it}, Adv. Math \textbf{329} (2018), 541--554.
%%\bibitem{BKW} K. Bringmann, B. Kane,  W. Kohnen, \begin{it}Locally harmonic Maass forms and the kernel of the Shintani lift\end{it}, Int. Math. Res. Not. \textbf{2015} (2015), 3185--3224.
%\bibitem{BKvP} K. Bringmann, B. Kane, and A. von Pippich, \begin{it}Regularized inner products of meromorphic modular forms and higher Green's functions\end{it}, Commun. Contemp. Math., accepted for publication.
%\bibitem{BKM}K. Bringmann, B. Kane,  M. Viazovska, \begin{it}Theta lifts and local Maass forms\end{it}, Math. Res. Lett. \textbf{20} (2013), 213--234.
%\bibitem{BringmannLovejoy}K. Bringmann and J. Lovejoy, \begin{it}Overpatitions and class numbers of binary quadratic forms\end{it}, Proc. Natl. Acad. Sci. USA \textbf{106} (2009), 5513--5516.
%\bibitem{BM}K. Bringmann and J. Manschot, \begin{it}From sheaves on $\mathbb{P}^2$ to a generalization of the Rademacher expansion\end{it}, Amer. J. Math. \textbf{135} (2013), 1039--1065.
%\bibitem{BM}K. Bringmann and K. Mahlburg, {\it Asymptotic formulas for stacks and unimodal sequences}, J. Comb. A \textbf{126} (2014), 194--215.
%\bibitem{BN}K. Bringmann and C. Nazaroglu, \begin{it}A framework for modular properties of false theta functions\end{it}, Res. Math. Sci., to appear.
%\bibitem{BRR}K. Bringmann, M. Raum, and O. Richter, \begin{it}Harmonic Maass--Jacobi forms with singularities and a theta-like decomposition\end{it}, Trans. Amer. Math. Soc. \textbf{367} (2015), 6647--6670.
\bibitem{BrockGranville}B. Brock and A. Granville, \begin{it}More points than expected on curves over finite field extensions\end{it}, Finite Fields and Appl. \textbf{7} (2001), 70--91.
\bibitem{BrownCalkin}B. Brown, N. Calkin, T. Flowers, K. James, E. Smith, and A. Stout, \begin{it}Elliptic curves, modular forms, and sums of Hurwitz class numbers\end{it}, J. Number Theory \textbf{128} (2008), 1847--1863.
%\bibitem{BruinierFunkeTraces}J. Bruinier and J. Funke, \begin{it}Traces of CM values of modular functions\end{it}, J. reine ange. Math. \textbf{594} (2006), 1--33.
%\bibitem{BruinierFunke} J. Bruinier and J. Funke, \begin{it}On two geometric theta lifts\end{it},  Duke Math. J.  \begin{bf}125\end{bf}  (2004),  no. 1, 45--90.
%\bibitem{BKO}J. Bruinier, W. Kohnen, and K. Ono, \begin{it}The arithmetic of the values of modular functions and the divisors of modular forms\end{it}, Compositio Math. \textbf{130} (2004), 552--566.
%\bibitem{123ModForms}J. Bruinier, G. van der Geer, G. Harder, and D. Zagier, \begin{it}The 1-2-3 of modular forms\end{it}, Springer, Berlin, 2009.
%\bibitem{BKO} J. Bruinier, W. Kohnen, and K. Ono, \begin{it}The arithmetic of the values of modular functions and the divisors of modular forms\end{it} Compositio Math. \textbf{130} (2004), 552--566.
\bibitem{CastryckHubrechts}W. Castryck and H. Hubrechts, \begin{it}The distribution of the number of points modulo an integer on elliptic curves over finite fields\end{it}, Ramanujan J. \textbf{30} (2013), 223--242.
%\bibitem{Choie}Y. Choie, \begin{it}Correspondence among Eisenstein series $E_{2,1}(\tau,z)$, $H_{\frac{3}{2}}(\tau)$ and $E_2(\tau)$,\end{it} Manuscripta Math. \textbf{93} (1997), 177--187.
%\bibitem{ChoieLee}Y. Choie and M. Lee, \begin{it}Correspondence among Eistein series of weight $2$,\end{it} J. Math. Anal. Appl. \textbf{344} (2008), 322--339.
\bibitem{CHT}L. Clozel, M. Harris, and R. Taylor, \begin{it}Automorphy for some $\ell$-adic lifts of automorphic mod $\ell$ Galois representations\end{it}, Publ. Math. Inst. Hautes \'Etudes Sci. \textbf{108} (2008), 1--181. 
\bibitem{Cohen}H. Cohen, \begin{it}Sums involving the values at negative integers of $L$-functions of quadratic characters\end{it}, Math. Ann. \textbf{217} (1975), 217--285.
%\bibitem{Cox}D. Cox, \begin{it}Primes of the form $x^2+ny^2$,\end{it} John Wiley and Sons, 1989.
%\bibitem{DMZ}A. Dabholkar, S. Murthy, and D. Zagier, \begin{it}Quantum black holes, wall crossing, and mock modular forms\end{it}, to appear in Cambridge Monographs in Mathematical Physics. 
\bibitem{DKS} C. David,  D. Koukoulopoulos, and E. Smith, \begin{it}Sums of Euler products and statistics of elliptic curves \end{it},  Math. Ann. \textbf{368} (2017),  685--752.
\bibitem{Deligne}P. Deligne, \begin{it}La conjecture de Weil I,\end{it} Inst. Hautes \'Etudes Sci. Publ. Math. \textbf{43} (1974), 273--307.
\bibitem{Deuring}M. Deuring, \begin{it}Die Typen der Multiplikatorenringe elliptische Funktionenk\"orper\end{it} Abh. Math. Sem. Hansischen Univ. \textbf{14} (1941), 197--272.
\bibitem{NIST} Digital Library of Mathematical Functions, National Institute of Standards and Technology, 
http://dlmf.nist.gov/.
\bibitem{DLZ} R Donepudi, J Li, A Zaharescu, \begin{it} Exact evaluation of second moments associated with some families of curves over a finite field, \end{it}, Finite Fields and Their Applications \textbf{48} (2017), 331--355.
\bibitem{Eichler}M. Eichler, \begin{it}On the class of imaginary quadratic fields and sums of divisors of natural numbers\end{it}, J. Indian Math. Soc. \textbf{19} (1956), 153--180.
\bibitem{EichlerZagier}M. Eichler and D. Zagier, \begin{it}The theory of Jacobi forms\end{it}, Progr. Math. \textbf{55}, Birkh\"auser, 1985.
%\bibitem{Fay} J. Fay, {\it{Fourier coefficients of the resolvent for a Fuchsian group}}, J. reine angew. Math. \textbf{293-294} (1977), 143--203.
%\bibitem{FreitagBusam}E. Freitag and R. Busam, \begin{it}Complex Analysis\end{it} (translated by Dan Fulea from the German text ``Funktionentheorie I''), Universitext, Springer--Verlag, 2005.
%\bibitem{Funke}J. Funke, \begin{it}CM points and weight $3/2$ modular forms\end{it} in Analytic Number Theory: a tribute to Gauss and Dirichlet, Clay Math. Proc. \textbf{7} (2007), 107--127.
\bibitem{GM} S. Galbraith and J. McKee, \begin{it} The probability that the number of points on an elliptic curve over a finite field is prime\end{it}, J. London Math. Soc. \textbf{62} (2000), 671--684.
\bibitem{Gekeler} E. Gekeler, \begin{it}Frobenius distributions of elliptic curves over finite prime fields,\end{it}, Int. Math. Res. Not. \textbf{2003}, no. 37, (2003), 1999--2018.
%\bibitem{Gierster}J. Gierster, \begin{it}\"Uber Relationen zwischen Klassenzahlen bin\"arer quadratischer Formen von negativer Determinante (Erste Abhandlung, Zweite Abhandlung)\end{it}, Math. Ann. \textbf{21} (1883), 1--50; \textbf{22} (1883), 190--210.
%\bibitem{Goldstein1}L. Goldstein, \begin{it}Dedekind sums for a Fuchsian group I,\end{it} Nagoya Math. J. \textbf{50} (1973), 21--47.
%\bibitem{Goldstein2}L. Goldstein, \begin{it}Dedekind sums for a Fuchsian group II,\end{it} Nagoya Math J. \textbf{53} (1974), 171--187.
\bibitem{GrossZagier} B. Gross, D. Zagier, \begin{it}Heegner points and derivatives of $L$-series\end{it}, Invent. Math. \textbf{84} (1986), 225--320.
%               \bibitem{Hecke}E. Hecke, \begin{it}Mathematische Werke\end{it}, p. 200.
\bibitem{Harris}M. Harris, {\it Galois representations, automorphic forms, and the Sato-Tate conjecture,} Indian J. Pure Appl. Math. 45 (2014), no. 5, 707--746.
\bibitem{HST}M. Harris, N. Shepherd-Barron, and R. Taylor, \begin{it}A family of Calabi--Yau varieties and potential automorphy\end{it}, Ann. Math. \textbf{171} (2010), 779--813.
%\bibitem{Hartung}P. Hartung, \begin{it}Proof of the existence of infinitely many imaginary quadratic fields whose class number is not divisible by $n$\end{it}, J. Number Theory \textbf{6} (1974), 276--278. 
%\bibitem{HM} J. Harvey, G. Moore, \begin{it} Algebras, BPS states, and strings\end{it}, Nuclear Phys. B \textbf{463} (1996), 315--368.

\bibitem{Hasse}
 H. Hasse, \begin{it}Zur Theorie der abstrakten elliptischen Funktionenk\"orper. I, II,  and III\end{it}, J. reine angew. Math. \textbf{175} (1936), 55--62, 69--88, 193--208.

 %\bibitem{Hejhal}D. Hejhal, \begin{it}The Selberg trace formula for $\SL_2(\R)$, Volume 2\end{it}, Lecture Notes in Mathematics \textbf{1001}, Springer--Verlag, 1983.
%\bibitem{HIvPT}S. Herrero, \"O. Imamo$\overline{\text{g}}$lu, A. M. von Pippich, and \'A. T\'oth, \begin{it}A Jensen--Rohrlich type formula for the hyperbolic 3-space\end{it}, preprint.

\bibitem{HZ}F. Hirzebruch and D. Zagier, \begin{it}Intersection numbers of curves on Hilbert modular surfaces and modular forms of Nebentypus\end{it}, Invent. Math. \textbf{36} (1976), 57-–113.
%\bibitem{HirzebruchZagier}F. Hirzebruch and D. Zagier, \begin{it}Class numbers, continued fractions, and the Hilbert modular group\end{it}, (to appear) {\bf Note: This paper never seems to have appeared.}
%\bibitem{Horie}K. Horie, \begin{it}A note on basic Iwasawa $\lambda$-invariants of imaginary quadratic fields\end{it}, Invent. Math. \textbf{88} (1987), 31--38.
%  \bibitem{ImOs}\"O. Imamo{$\overline{\text{g}}$}lu and C. O'Sullivan, \begin{it}Parabolic, hyperbolic and elliptic Poincar\'e series\end{it}, Acta Arith. \textbf{139} (2009), 199--228.
\bibitem{Ihara}Y. Ihara, \begin{it}Hecke polynomials as congruence $\zeta$ functions in elliptic modular case\end{it}, Ann. Math. \textbf{85} (1967), 267--295.
%\bibitem{JetchevKane}D. Jetchev and B. Kane, \begin{it}Equidistribution of Heegner poitns and ternary quadratic forms\end{it}, Math. Ann. \textbf{350} (2011), 501--532. 
\bibitem{KaplanPetrow2}N. Kaplan and I. Petrow, \begin{it}Elliptic curves over a finite field and the trace formula\end{it}, Proc. London Math. Soc. \textbf{115} (2017), 1317--1372.
\bibitem{KaplanPetrow}N. Kaplan and I. Petrow, \begin{it}Traces of Hecke operators and refined weight enumerators of Reed--Solomon codes\end{it}, Trans. Amer. Math. Soc. \textbf{370} (2018), 2537--2561.
\bibitem{KatzSarnak}N. Katz and P. Sarnak, \begin{it} Random matrices, Frobenius eigenvalues, and monodromy\end{it}, AMS Coll. Publ. \textbf{45}, AMS Math. Soc. , Providence, 1999.
\bibitem{Koblitz}N. Koblitz, \begin{it}Introduction to elliptic curves and modular forms\end{it}, Graduate texts in Math. \textbf{97}, Springer-Verlag, 1993.
%\bibitem{KoecherKrieg}M. Koecher and A. Krieg, \begin{it}Elliptische Funktionen und Modulformen\end{it}, Springer, Berlin, 2007.
%\bibitem{KohnenJacobi}W. Kohnen, \begin{it}Non-holomorphic Poincar\'e-type series on Jacobi groups\end{it}, J. Number Theory \textbf{46} (1994), 70--99.
%\bibitem{KohnenOno}W. Kohnen and K. Ono, \begin{it}Indivisibility of class numbers of imaginary quadratic fields and orders of Tate--Shafarevich groups of elliptic curves with complex multiplication\end{it}, Invent. Math. \textbf{135} (1999), 387--398.
\bibitem{Kowalski}E. Kowalski, \begin{it}Analytic problems for elliptic curves\end{it}, J. Ramanujan Math. Soc. \textbf{21} (2006), 19--114.
%\bibitem{Kronecker}L. Kronecker, \begin{it}\"Uber die Anzahl der verschiedenen Klassen quadratischer Formen von negativer Determinante\end{it}, J. reine Angew. Math. \textbf{57} (1860), 248--255. 
%\bibitem{Kudla}S. Kudla, \begin{it}Integrals of Borcherds forms\end{it}, Compositio Math. \textbf{137} (2003), 293--349.
%\bibitem{Kuhn} U. K\"uhn, \begin{it}Generalized arithmetic intersection numbers\end{it}, J. reine Angew. Math. \textbf{534} (2001), 209--236.
%\bibitem{LachaudWolfmann}G. Lachaud and J. Wolfmann, \begin{it}The weights of orthogonals of the extended quadratic binary Goppa codes\end{it}, IEEE Trans. Inform. Theory \textbf{36} (1990), 686--692. 
\bibitem{Lavrik} A. F. Lavrik, \begin{it}On the problem of distribution of the values of class number of properly primitive quadratic forms with negative determinant\end{it}, Izv. Akad. Nauk UzSSR Ser. Fiz.-Mat. \textbf{1} (1959), 81--90.
%\bibitem{Lenstra}H. Lenstra, Jr. \begin{it}Factoring integers with elliptic curves\end{it}, Ann. of Math. (2) \textbf{126} (1987), 649--673.
%\bibitem{Li} W. Li, \begin{it}Newforms and functional equations\end{it}, Math. Ann. \textbf{212} (1975), 285--315.
\bibitem{Littlewood} J. Littlewood, \begin{it}On the class number of the corpus $P(\sqrt{-k})$\end{it}, Proc. London Math. Soc. \textbf{27} (1928), 358--372.
%\bibitem{Loebrich}S. L\"obrich, \begin{it}Niebur Poincar\'e series and traces of singular moduli\end{it}, submitted for publication.
%\bibitem{Matsusaka}T. Matsusaka, \begin{it}Traces of CM values and cycle integrals of polyharmonic Maass forms\end{it}, preprint.
\bibitem{McKee} J. McKee, \begin{it} Subtleties in the distribution of the numbers of points on elliptic curves over a finite prime field \end{it}, J. Lond. Math. Soc. \textbf{59} (1999), 448--460.
%\bibitem{Meyer} C. Meyer, \begin{it}\"Uber die Berechnung der Klassenzahl abelscher K\"orper \"uber quadratischen Zahlenk\"orper\end{it}, Akademie-Verlag, Berlin, 1957.
\bibitem{Me} M.\,Mertens, {\it Eichler-Selberg type identities for mixed mock modular forms}, Adv. Math. \textbf{301} (2016), 359--382.
\bibitem{MOR} M. Mertens, K. Ono, and L. Rolen, {\it Mock Modular Eisenstein series with Nebentypus}, preprint.
%\bibitem{Moisio} M. Moisio, \begin{it}On the moments of Kloosterman sums and fibre products of Kloosterman curves\end{it}, Finite Fields and their Applications \textbf{14} (2008), 515--531.

%\bibitem{Murty-Murty} M. Murty and V. Murty, \begin{it}Some remarks on automorphy and the Sato-Tate conjecture\end{it}, in: Advances in the theory of numbers, eds. A. Alaca et al. pp. 159--168, Fields Institute Communications Volume {\bf 77}, Springer, New York, 2015.
%\bibitem{Niebur}D. Niebur, \begin{it}A class of nonanalytic automorphic functions\end{it}, Nagoya Math. J. \textbf{52} (1973), 133--145.
\bibitem{OnoBook}K. Ono, \begin{it}The web of modularity: arithmetic of the coefficients of modular forms and $q$-series\end{it}, CMBS Regional Conference Series in Mathematics \textbf{102} (2004), American Mathematical Society, Providence, RI, USA.  
%\bibitem{OnoSkinner}K. Ono and C. Skinner, \begin{it}Fourier coefficients of half-integral weight modular forms modulo $\ell$\end{it}, Ann. Math \textbf{147} (1998), 453--470.
%\bibitem{PeEinheit}H. Petersson, \begin{it}Einheitliche Begr\"undung der Vollst\"andigkeitss\"atze f\"ur di Poincar\'eschen Reihen von reeller Dimension bei beliebigen Grenzkreisgruppen von erster Art\end{it}, Abh. math. Sem. Hansischen Univ. \textbf{14} (1941), 22--60.
%\bibitem{Pe} H. Petersson, \begin{it}Ein Summationsverfahren f\"ur die Poincar\'eschen Reihen von der Dimension 2 zu den hyperbolischen Fixpunkten\end{it}, Math. Z. \textbf{49} (1943), 441--496.

%\bibitem{Pe1} H. Petersson, \begin{it}Konstruktion der Modulformen und der zu gewissen Grenzkreisgruppen geh\"origen automorphen Formen von positiver reeller Dimension und die vollst\"andige Bestimmung ihrer Fourierkoeffizienten\end{it}, 1950.
%\bibitem{Pe2} H. Petersson, \begin{it}\"Uber automorphe Orthogonalfunktionen und die Konstruktion der automorphen Formen von positiver reeller Dimension\end{it}, Math. Ann. \textbf{127} (1954), 33--81.!$
%               \bibitem{PeterssonProperties} H. Petersson, \begin{it}The properties of the representation of the Abelian differentials by Poincar\'e series\end{it},
%\bibitem{Ramanujan}S. Ramanujan, \begin{it}On certain arithmetical functions\end{it}, Trans. Cambridge Phil. Soc. \textbf{22} (1916), 159--184.
%\bibitem{Rohrlich}D. Rohrlich, \begin{it}A modular version of Jensen's formula\end{it}, Math. Proc. Cambridge Phil. Soc. \textbf{95} (1984), 15--20.
%\bibitem{RosserSchoenfeld}J. Rosser and L. Schoenfeld, \begin{it} Approximate formulas for some functions of prime numbers\end{it}, Ill. J. Math. \textbf{6} (1962), 64--94.
%\bibitem{Liesuper}  A. Semikhatov, A. Taormina, and I. Tipunin, \begin{it}Higher-level Appell functions, modular transformations, and characters\end{it}, Comm. Math. Phys. \textbf{255} (2005), 469--512.
%\bibitem{Saparnijazov} O. Saparnijazov, 
\bibitem{Sarnak} P. Sarnak, \begin{it}Statistical properties of eigenvalues of the Hecke operators\end{it}, Analytic Number Theory and Diophantine Problems, Progr. Math. \textbf{70}, Stillwater, 1984, Birkh\"auser, Basel (1987), 321--331.
\bibitem{Schoof}R. Schoof, \begin{it}Nonsingular plane cubic curves over finite fields\end{it}, J. Comb. Theory Ser. A \textbf{46} (1987), 183--211.
\bibitem{SerreStark}J.-P. Serre and H. Stark, \begin{it}Modular forms of weight $\frac{1}{2}$\end{it}, in Modular functions of one variable VI, Lecture notes in Math. \textbf{627} (1977), Springer, Berlin, 27--67.
\bibitem{Serre} J.-P. Serre, \textit{R\'{e}partition asymptotique des valeurs propres de l'op\'{e}rateur de Hecke $T_p$}, J. Amer. Math. Soc. \textbf{10} (1997), 75--102. 
%\bibitem{Shimura}G. Shimura, \begin{it}On modular forms of half integral weight\end{it}, Ann. Math. \textbf{97} (1973), 440--481.

\bibitem{Siegel}C. Siegel, \begin{it}\"Uber die Classenzahl quadratischer Zahlk\"orper\end{it}, Acta. Arith. \textbf{1} (1935), 83--86.
\bibitem{Silverman}J. Silverman, \begin{it}The arithmetic of elliptic curves\end{it}, Graduate texts in Math. \textbf{106}, Springer-Verlag, 2009.
\bibitem{SZ} I. E. Shparlinski and L. Zhao, \begin{it} Elliptic curves in isogeny classes \end{it}, J. Number Theory \textbf{191} (2018), 194--212. 
\bibitem{Sturm}J. Sturm, \begin{it}Projections of $C^{\infty}$ automorphic forms,\end{it} Bull. Amer. Math. Soc. \textbf{2} (1980), 435--439. 
%\bibitem{TW}R. Taylor and A. Wiles, \begin{it}Ring theoretic properties of certain Hecke algebras,\end{it} Ann. Math. \textbf{141} (1995), 553--572.
%\bibitem{Vinogradov}
%\bibitem{Waterhouse}W. Waterhouse, \begin{it}Abelian varieties over finite fields\end{it} Ann. Sci. \'Ecole Norm. Sup \textbf{2} (1969), 521--560.
%\bibitem{Wiles}A. Wiles, \begin{it}Modular elliptic curves and Fermat's Last Theorem\end{it}, Ann. Math. \textbf{141} (1995), 443--551.
%\bibitem{Williams}B. Williams, \begin{it}Vector-valued Hirzebruch--Zagier series and class number sums\end{it}, Res. Math. Sci. \textbf{5:25} (2018).

\bibitem{Wolke1}D. Wolke, \begin{it}Moments of the number of classes of primitive quadratic forms with negative discriminant\end{it}, J. Number Theory \textbf{1} (1969), 502--511. 
\bibitem{Wolke2}D. Wolke, \begin{it}Momente der klassenzahlen, II\end{it}, Arch. Math. \textbf{22} (1971), 65--69.
\bibitem{Wolke3}D. Wolke, \begin{it}Moments of class numbers, III\end{it}, J. Number Theory \textbf{4} (1972), 523--531.

%\bibitem{ZagierUtrecht}D. Zagier, \begin{it}Introduction to modular forms\end{it} in ``From Number Theory to physics'', Springer Verlag, Heidelberg (1992), 238--291.
%\bibitem{Zagier}D. Zagier, \begin{it}Nombres de classes et fractions continues\end{it}, Ast\'erisque \textbf{24-25} (1975), 81--97.
%\bibitem{ZagierZeta}D. Zagier, \begin{it}Zetafunktionen und quadratische K\"orper: Eine einf\"uhrung in die h\"ohere Zahlentheorie\end{it}, Springer-Verlag, 1981.
%\bibitem{ZagierSingular}D. Zagier, \begin{it}Traces of singular moduli\end{it}, in ``Motives, polylogarithms and Hodge Theorem (eds. F. Bogomolov and L. Katzarkov), Lecture series \textbf{3} (2002), 209--244.
%\bibitem{ZwegersAppell}S. Zwegers, \begin{it}Multivariable Appell functions and nonholomorphic Jacobi forms\end{it}, Res. Math. Sci. {\bf 6} (2019). 

\end{thebibliography}
\end{document}